\documentclass{amsart}

\usepackage{amsmath,amssymb,url}

\usepackage{verbatim}

\numberwithin{equation}{section}

\newtheorem{thm}{Theorem}  
\newtheorem{prop}[thm]{Proposition} 
\newtheorem{lem}[thm]{Lemma} 
\newtheorem{cor}[thm]{Corollary} 
\newtheorem{fact}[thm]{Fact}

\theoremstyle{definition}   
\newtheorem{defn}[thm]{Definition} 
\newtheorem{question}[thm]{Question}
\newtheorem{problem}[thm]{Problem}

\newcommand{\al}{\alpha}
\newcommand{\om}{\omega}

\newcommand{\fin}{\mathrm{fin}}

\newcommand{\bo}{\mathbf{0}}
\newcommand{\bone}{\mathbf{1}}

\newcommand{\sse}{\subseteq}
\newcommand{\contains}{\supseteq}

\DeclareMathOperator{\cf}{cf}
\DeclareMathOperator{\ci}{ci}
\DeclareMathOperator{\Clop}{Clop}
\DeclareMathOperator{\ro}{r.o.}

\newcommand{\bv}{\bigvee}
\newcommand{\bw}{\bigwedge}

\newcommand{\rgl}{\rangle}
\newcommand{\lgl}{\langle}
\newcommand{\re}{\restriction}

\newcommand{\bB}{\mathbb{B}}

\newcommand{\lra}{\leftrightarrow}
\newcommand{\ra}{\rightarrow}

\newcommand{\n}{\noindent}
\newcommand{\abs}[1]{\lvert#1\rvert}
\newcommand{\treealg}{\operatorname{Treealg}}
\newcommand{\intalg}{\operatorname{Intalg}}
\newcommand{\cn}{\uparrow}
\newcommand{\bl}{\downarrow}
\newcommand{\ult}{\operatorname{Ult}}  
\newcommand{\imm}{\operatorname{imm}}
\newcommand{\topkappa}{([\kappa]^{<\omega},\subseteq)}   
\newcommand{\U}{\mathcal{U}}
   
\theoremstyle{remark} 
\newtheorem*{rem}{Remark} 
\newtheorem*{claim}{Claim}
\newtheorem{subclaim}{Subclaim}

\begin{document}   


\title[Spectra of Tukey types]{Spectra of Tukey types of ultrafilters on Boolean algebras}

\author[J. A. Brown]{Jennifer A. Brown}
\email{jennifer.brown@csuci.edu}
\address{Department of Mathematics\\  California  State University - Channel Islands\\ One University Drive \\ Camarillo, CA 93012 U.S.A.}

\author[N. Dobrinen]{Natasha
  Dobrinen}
\email{natasha.dobrinen@du.edu}
\urladdr{http://web.cs.du.edu/~ndobrine}
\address{Department of Mathematics\\
  University of Denver \\
   2280 S Vine St\\ Denver, CO \ 80208 U.S.A.}

\thanks{Dobrinen was supported by  National Science Foundation Grant DMS-1301665.}



\begin{abstract}
Extending recent investigations on  the  structure  of  Tukey types of ultrafilters on $\mathcal{P}(\om)$ to Boolean algebras in general, we classify the  spectra of Tukey types of ultrafilters for several classes of Boolean algebras, including interval  algebras, tree algebras,  and  pseudo-tree algebras.
\end{abstract}

\maketitle


\section{Introduction}\label{sec.intro}

The structure of Tukey types of ultrafilters on $\om$ is the subject  of much recent and ongoing research. 
This line of research was  reignited, after Isbell's study in \cite{Isbell65}, by   Milovich in \cite{Milovich08}.
One of the interests in this research stems from the connection via Stone duality  between Tukey types of ultrafilters on $\om$ and cofinal types of neighborhood bases on the \v{C}ech-Stone compactification $\beta\om$ of the natural numbers and its \v{C}ech-Stone remainder $\beta\om\setminus\om$.
Recently, the structure of the Tukey types of ultrafilters on $\om$ has undergone an explosion of activity, starting with \cite{Dobrinen/Todorcevic11} of
Dobrinen and Todorcevic, 
and continuing in
work of Milovich in \cite{Milovich12},  Raghavan and Todorcevic in \cite{Raghavan/Todorcevic12},
further work of Dobrinen and Todorcevic in
\cite{Dobrinen/Todorcevic14}, \cite{Dobrinen/Todorcevic15},
Blass, Dobrinen, and Raghavan in \cite{Blass/Dobrinen/Raghavan13},
Dobrinen, Mijares, and Trujillo in \cite{Dobrinen/Mijares/Trujillo14}, and most recently Dobrinen in \cite{Dobrinen14}.
These lines of study have all been focused  on clarifying the structure of Tukey types of ultrafilters on $\om$.
In particular, the aforementioned works provide good understanding of the Tukey structure  for a large class of  ultrafilters possessing some means of diagonalization, in particular, p-points.
The reader is referred to the recent survey article \cite{Dobrinen13} for more background.

There remain important open problems regarding Tukey types of ultrafilters on $\om$.
One of the most prominent open problems is due to Isbell.
In \cite{Isbell65}, Isbell showed that there are ultrafilters on $\om$ which have the maximum Tukey type, namely $([\mathfrak{c}]^{<\om},\sse)$.
He then asked whether there are ultrafilters on $\om$ which do not have the maximum Tukey type.
By work in \cite{Dobrinen/Todorcevic11}, p-points are never Tukey maximum.  However, Shelah constructed   a model of ZFC in which there are no p-points (see \cite{ProperForcingBK}),
so what remains of Isbell's problem is the following.
\begin{problem}\label{problem.Isbell}
Is there a model of ZFC in which all nonprincipal ultrafilters have the maximum Tukey type?
\end{problem}

Milovich showed in \cite{Milovich08} that there is a nonprincipal ultrafilter $\mathcal{U}$ on $\omega$ such that $(\mathcal{U},\contains)$ is not Tukey maximum if and only if there is a nonprincipal  ultrafilter $\mathcal{V}$ on $\omega$ such that $(\mathcal{V},\contains^*)$ is not Tukey maximum.
Thus,  Isbell's problem for ultrafilters on the Boolean algebra $\mathcal{P}(\om)$ is equivalent to the  problem for ultrafilters on the Boolean algebra 
$\mathcal{P}(\om)/ \fin$.

In this paper,
we  expand the investigation of the Tukey structure of ultrafilters to the general class of all Boolean algebras.
The main motivation  is  the analogue of  Isbell's problem for Boolean algebras in general:
\begin{problem}\label{q.IsbellBA}
Characterize the class of Boolean algebras  for which all its nonprincipal ultrafilters have the maximum Tukey type.
Further, characterize the class of Boolean algebras which have a nonprincipal ultrafilter which is not of maximum Tukey type.
\end{problem}
By understanding the characteristics of  Boolean algebras for which we can prove in ZFC whether or not it has more than one Tukey type of nonprincipal ultrafilter, 
we hope to gain new insight and methods for solving Isbell's problem on $\mathcal{P}(\omega)$ and $\mathcal{P}(\om)/\fin$.

The second motivation for this line of work is to use the structure of Tukey types of ultrafilters as a means of classifying Boolean algebras.
Tukey theory has been successfully used to classify  partial orders.
This is especially important in  cases when isomorphism is too strong a notion to gain any meaningful information.
The seminal works of \cite{TodorcevicDirSets85}, \cite{Todorcevic96}, and \cite{Solecki/Todorcevic04} used Tukey reducibility to classify 
directed, undirected, and analytic partial orders, respectively.
In a similar vein, the structure of  the Tukey types of ultrafilters on  Boolean algebras can be used as a means for a classifying Boolean algebras.
Given a Boolean algebra $\bB$, it follows from  Stone duality that
the Tukey type of an ultrafilter $\mathcal{U}$ on $\bB$ is the same as the cofinal type of the neighborhood basis of $\mathcal{U}$ in the Stone space of $\bB$.
Thus, the study of Tukey types of ultrafilters on Boolean algebras is equivalent to the study of cofinal types of neighborhood bases of Boolean spaces, that is, zero-dimensional compact Hausdorff spaces.
Those  Boolean algebras with similar Tukey spectra have similar topological properties in their Stone spaces, as the collection of all Tukey types of ultrafilters on a Boolean algebra  is equal to  the  collection of all cofinal types of  neighborhood bases of points in the Stone space.

Before stating more precise questions and basic facts,
we review the relevant definitions.
A partial order $(P,\le)$ is {\em directed} if for any two $p,q\in P$, there is an $r\in P$ such that $p\le r$ and $q\le r$.
An ultrafilter on a Boolean algebra $\bB$ is  simply a maximal filter on $\bB$.
Letting $\le$ denote the natural partial ordering on a Boolean algebra $\bB$ (defined by $a\le b$ if and only if $a\wedge b=a$),
we point out that  for any ultrafilter $\mathcal{U}$ on a Boolean algebra $\bB$,
$(\mathcal{U},\ge)$ is a directed partial ordering.
Although Tukey reducibility can be defined for transitive relations in general, we shall assume here that all partial orders discussed are directed,  since this article concentrates on ultrafilters.

Given $(P,\le_P)$, 
a subset $X\sse P$ is {\em cofinal} in $P$ if for each $p\in P$, there is an $x\in X$ such that $p\le_P  x$.
A  map $f \colon (P,\le_P)\ra (Q,\le_Q)$ is {\em cofinal} if for each cofinal $X\sse P$, the image $f[X]$ is cofinal in $Q$.
A set $X\sse P$ is {\em unbounded} if there is no upper bound for $X$ in $P$; that is, there is no
$p\in P$ such that for each $x\in X$, $x\le_P p$.
A map $g \colon (P,\le_P)\ra(Q,\le_Q)$ is {\em unbounded} or {\em Tukey}  if for each unbounded $X\sse P$, the image $g[X]$ is unbounded in $Q$. 
Schmidt showed  in \cite{Schmidt55} that the existence of a cofinal map from $(P,\le_P)$ to $(Q,\le_Q)$  is equivalent to the existence of a Tukey map from $(Q,\le_Q)$ to $(P,\le_P)$.
We say that $(P,\le_P)$ is {\em Tukey reducible to} $(Q,\le_Q)$, and write   $(P,\le_P)\le_T (Q,\le_Q)$, if there is  a cofinal map from $(Q,\le_Q)$ to $(P,\le_P)$ (\cite{Tukey40}).
We say that 
$(P,\le_P)$ is {\em Tukey equivalent to} $(Q,\le_Q)$, and write $(P,\le_P)\equiv_T (Q,\le_Q)$, if and only if
$(P,\le_P)\le_T (Q,\le_Q)$ and $(Q,\le_Q)\le_T (P,\le_P)$.

It follows from work of Schmidt in \cite{Schmidt55} and is implicit in work of Tukey \cite{Tukey40} that, for each infinite cardinal $\kappa$, 
the directed partial ordering  $([\kappa]^{<\om},\sse)$  is the maximum among all Tukey types of directed partial orderings of cardinality $\kappa$.
That is, every directed partial ordering of cardinality $\kappa$ is Tukey reducible to  $([\kappa]^{<\om},\sse)$. The minimal Tukey type is 1 (the one-element partially ordered set). An ultrafilter partially ordered by $\geq$ has this minimal Tukey type if and only if it is principal. For directed partial orderings,  Tukey types and cofinal types coincide.
(See \cite{TodorcevicDirSets85} for more background.)

An ultrafilter $\mathcal{U}$ on a Boolean algebra $\bB$ has the {\em maximum Tukey type}
if  $(\mathcal{U},{\ge})\equiv_T([|\bB|]^{<\om},{\sse})$, which is the maximum possible cofinal type for a directed partial ordering of cardinality $|\bB|$.
Isbell's result in \cite{Isbell65} that there is an ultrafilter on $\om$ which has maximum Tukey type 
 led us to ask whether this is true for   Boolean algebras in general:
Given a Boolean algebra, is there an ultrafilter which has the maximum Tukey type?
We will show  that this is not the case: there are several classes of Boolean algebras which do not have any ultrafilters with the maximum Tukey type, the uncountable interval algebras being the simplest of these (see Theorem \ref{interval algebras}).
This along with Problem \ref{q.IsbellBA}   led to the following collection of questions which we investigate.
Our intuition was that  Boolean algebras with strong structure would yield simple Tukey structures.

\begin{question}\label{q.existsmax}
For which Boolean algebras  is there an ultrafilter which has the maximum Tukey type?
\end{question}

\begin{question}\label{q.onlymax}
Are there Boolean algebras on which all ultrafilters have the maximum Tukey type?
\end{question}

\begin{question}\label{q.existsnonmax}
Which Boolean algebras have ultrafilters with non-maximum Tukey type?
\end{question}

\begin{question}\label{q.onlynonmax}
Can we characterize those   Boolean algebras  all of whose ultrafilters have Tukey type strictly below the maximum?
\end{question}

We answer these questions for certain classes of Boolean algebras by obtaining the finer results of classifying their Tukey spectra.
By the {\em Tukey  spectrum} of a Boolean algebra $\bB$, denoted Ts$(\bB)$, we mean the  collection of all Tukey types of  ultrafilters on $\bB$ partially ordered by Tukey reduction.
We use the terminology {\em Tukey spectra} of a class $\mathcal{C}$ of Boolean algebras to mean the collection  $\{$Ts$(\bB):\bB\in\mathcal{C}\}$.
The following is the main focus of this paper.

\begin{problem}\label{q.Ts}
Given a Boolean algebra, find its Tukey spectrum.
\end{problem}

We solve  Problem \ref{q.Ts} for the classes of  free Boolean algebras,  superatomic Boolean algebras generated by almost disjoint families, interval  algebras, tree algebras,  and  pseudo-tree algebras.

Section \ref{sec.top} 
concentrates on classes of Boolean algebras which  have an ultrafilter with maximum Tukey type.
We present  conditions guaranteeing the existence of an ultrafilter with the maximum Tukey type, providing some answers to  Question \ref{q.existsmax}.
In particular, 
 Theorem \ref{thm.indepimpliestop} shows 
that every infinite Boolean algebra with  an independent family of maximum cardinality has an ultrafilter with maximum Tukey type.
In Fact \ref{prop.freealgchar},
we point out that the class of free Boolean algebras answers Question \ref{q.onlymax} positively.
It is then natural to ask whether the completion of a free Boolean algebra has Tukey spectrum consisting of only the maximum Tukey type.
In Theorem \ref{thm.cohennot}  we rule out certain non-maximum Tukey types:  no finite product of 
regular cardinals is in the Tukey spectrum of the completion of a free Boolean algebra on infinitely many generators.
In
Theorem \ref{thm.adfamily} we find the Tukey spectra of superatomic Boolean algebras generated by almost disjoint families.
In particular,  the minimum and maximum Tukey types are always attained.

In Section \ref{sec.IntAlg}, we  classify the Tukey spectra  of the families of interval algebras, tree algebras, and pseudo-tree algebras.
Theorem \ref{interval algebras}
classifies the Tukey spectrum of  any interval algebra as a collection of products of two cardinals which are determined by adjutting initial and coinitial chains in the linear order.
It follows that 
the Tukey spectrum of the interval algebra of any uncountable linear order  does not include the maximum Tukey type.
Theorem \ref{tree algebras} characterizes the Tukey types of ultrafilters in tree algebras as the product of the cofinality of an initial chain in the tree with the weak product of the cardinality of its set of immediate successors.
This immediately leads to Theorem \ref{c.spectrum of tree algebra} which classifies the Tukey spectra of tree algebras.
Lemma \ref{lem.lambdafan}
distills the essential structure of a set of approximate immediate successors of an initial chain in terms of a $\lambda$-fan (see Definition \ref{def.lambdafan}).
This is used to 
 classify  the Tukey spectra of pseudo-tree algebras
in Theorem \ref{p.when pseudo-tree algebras have top type ult} in terms of the product of an initial chain $C$ and the weak product of the cardinals determined by the $\lambda$-fan of approximate immediate successors above $C$.

In addition to finding the Tukey spectra for these classes of Boolean algebras,
we 
show that for certain prescribed sets of partial orders, a Boolean algebra can be constructed which contains those partial orders in its Tukey spectrum. 
This is done for interval algebras in 
Fact \ref{f.attainable types for interval algebras},
and for  pseudo-tree algebras in 
Proposition \ref{finite product types attained}, which also takes care of tree algebras.

We conclude the introduction with some notation and basic facts which will be useful throughout the paper.
A function $f \colon (P,\le_P)\ra (Q,\le_Q)$ is called {\em monotone} if whenever $p\le_P p'$ then $f(p)\le_Q f(p')$.

\begin{fact}\label{fact.mono}
\begin{enumerate}
\item
Monotone maps with cofinal images are cofinal maps.
\item
Isomorphic partial orders have the same Tukey type.
\end{enumerate}
\end{fact}

For $\kappa, \mu$ cardinals, the partial order $\leq$ on $\kappa \times \mu$ is defined coordinate-wise:\ $(\alpha, \beta) \leq (\alpha', \beta')$ if and only if $\alpha \leq \alpha'$ and $\beta \leq \beta'$. 
More generally, we fix the following notation.
\begin{defn}\label{def.productofcards}
Given a collection of cardinals $\{\kappa_i:i\in I\}$ for some index set $I$,
  $\prod_{i\in I}\kappa_i$ denotes the collection of all functions $f \colon I\ra \bigcup\{\kappa_i:i\in I\}$ such that for each $i\in I$, $f(i)\in\kappa_i$.
The partial ordering $\le$ on  $\prod_{i\in I}\kappa_i$ is coordinate-wise:
 For $f,g\in  \prod_{i\in I}\kappa_i$, $f\le g$ if and only if for all $i\in I$, $f(i)\le g(i)$.
\end{defn}

\begin{fact}\label{facts.useful}
\begin{enumerate}
\item
For any infinite cardinal $\kappa$,
every directed partial ordering of cardinality less than or equal to $\kappa$ is Tukey reducible to 
$([\kappa]^{<\om},\sse)$.
\item
Let $\mathcal{U}$ be an ultrafilter on  a Boolean algebra $\bB$, and let $\kappa=|\bB|$.
Then $(\mathcal{U},\ge)\equiv_T([\kappa]^{<\om},\sse)$
if and only if
there is a subset $\mathcal{X}\sse\mathcal{U}$ of cardinality $\kappa$ such that for each infinite $\mathcal{Y}\sse\mathcal{X}$,
$\mathcal{Y}$ is unbounded in $\mathcal{U}$. 
\item
$(\om,\le)\equiv_T([\om]^{<\om},\sse)$.
\item
For any uncountable cardinal $\kappa$, $\topkappa >_T (\kappa, {\leq})$.
\item
For any infinite cardinal $\kappa$ and any $n<\omega$, $(\kappa, \leq) \equiv_T (\prod_{i<n}\kappa, \leq)$.
\item Let $\mathcal{U}$ be an ultrafilter on  a Boolean algebra $\bB$, and let $G \subseteq \bB$ be a filter base for $\U$. Then $(\U, \supseteq) \equiv_T (G, \supseteq)$.
\end{enumerate}
\end{fact}

\begin{proof}
(1) follows from work of Schmidt in \cite{Schmidt55}.
The proof of (2) is very similar to the proof of Fact 12 in \cite{Dobrinen/Todorcevic11}.
(3) is due to 
Day in \cite{Day44}.
(4) follows from (1) along with the fact that every countably infinite subset of an uncountable cardinal $\kappa$ is bounded in $\kappa$.
To show (5),
let $\kappa \geq \omega$ and $n<\omega$. Define $f \colon (\kappa, \leq) \to (\prod_{i<n}\kappa, \leq)$ by $f(\alpha)=\langle \alpha, \dots, \alpha \rangle$ (the element of $(\prod_{i<n}\kappa, \leq)$ constant at $\alpha$). Then $f$ is an unbounded cofinal map. For (6), note that a filter base $G$ for $\U$ is a cofinal subset of $\U$. Then (6) follows by Fact 3 in \cite{Dobrinen/Todorcevic11}.
\end{proof}

\section{Boolean algebras attaining the maximum Tukey type}\label{sec.top}

It is well-known that for each infinite  cardinal $\kappa$,
there is an ultrafilter on $\mathcal{P}(\kappa)$ with maximum Tukey type. 
(This follows from combining work of Isbell in \cite{Isbell65} and Schmidt in \cite{Schmidt55}.)
Such an ultrafilter may be constructed using an independent family on $\kappa$ of cardinality $2^{\kappa}$.
We begin by showing that this construction generalizes to any infinite Boolean algebra $\bB$ with an independent family of cardinality $|\bB|$.

\begin{thm}\label{thm.indepimpliestop}
If $\bB$ is an infinite  Boolean algebra with an independent family of cardinality $|\bB|$,
then there is an ultrafilter $\mathcal{U}$ on $\bB$ such that $(\mathcal{U},\ge)\equiv_T([|\bB|]^{<\om},\sse)$.
\end{thm}

\begin{proof}
Let $\bB$ be any infinite  Boolean algebra
 with an
 independent family $\mathcal{I}=\{a_{\al}:\al<|\bB|\}$, and let $\kappa$ denote $|\bB|$.
Define $\mathcal{J}=\{-b:b\in \bB$ and $\{\al<\kappa: b\le a_{\al}\}$ is infinite$\}$.
We  first show that $\mathcal{I}\cup\mathcal{J}$ has the finite intersection property.
Let $m,n<\om$, $\al_i$ $(i\le m)$ be distinct members of $\kappa$, and let
$b_j$ $(j\le n)$ be distinct members of $\mathcal{J}$.
Since each $b_j$ is below $a_{\al}$ for infinitely many $\al<\kappa$,
we may choose distinct $\beta_j$ ($j\le n$) such that for each $j\le n$,  $b_j\le a_{\beta_j}$,
and moreover $\beta_j\not\in \{\al_i:i\le m\}$.
Since $\mathcal{I}$ is independent it follows that $\bw_{i\le m} a_{\al_i}\wedge\bw_{j\le n}(-a_{\beta_j})>\bo$.
Since each $b_j\ge -a_{\beta_j}$,
we have that $\bw_{i\le m} a_{\al_i}\wedge\bw_{j\le n}b_j>\bo$.

Let $\mathcal{F}$ be the filter generated by $\mathcal{I}\cup\mathcal{J}$.
$\mathcal{F}$ is a proper filter, since its generating set has the finite intersection property.

\begin{claim} $(\mathcal{F},\ge)\equiv_T([\kappa]^{<\om},\sse)$.
\end{claim}

\begin{proof}
It  suffices to show that
 $(\mathcal{F},\ge)\ge_T([\kappa]^{<\om},\sse)$,
 since  $\mathcal{F}$ being a directed partial order of size $\kappa$ implies that $(\mathcal{F},\ge)\le_T([\kappa]^{<\om},\sse)$.
Define $f \colon [\kappa]^{<\om}\ra\mathcal{F}$ by $f(F)=\bw_{\al\in F}a_{\al}$, for each $F\in[\kappa]^{<\om}$.
We claim that $f$ is a Tukey map.
To see this, let $\mathcal{X}\sse[\kappa]^{<\om}$ be unbounded in $([\kappa]^{<\om},\sse)$.
Then $\mathcal{X}$ must be infinite.
The $f$-image of $\mathcal{X}$ is $\{\bw_{\al\in F}a_{\al}:F\in\mathcal{X}\}$.
Any lower bound $b$ of $\{\bw_{\al\in F}a_{\al}:F\in\mathcal{X}\}$ would have to have the property that
 $b\le\bw_{\al\in\bigcup\mathcal{X}}a_{\al}$.
But $\bv_{\al\in\bigcup\mathcal{X}}-a_{\al}$ is in $\mathcal{B}$, since $\bigcup\mathcal{X}$ is  infinite; so  the complement of $\bw_{\al\in\bigcup\mathcal{X}}a_{\al}$ is in $\mathcal{F}$.
Since $\mathcal{F}$ is a proper filter, 
$\mathcal{F}$ contains no members below
$\bw_{\al\in\bigcup\mathcal{X}}a_{\al}$.
It follows that 
 the $f$ image of $\mathcal{X}$ is unbounded in $\mathcal{F}$.
Therefore $f$ is a Tukey map and the claim holds.
\end{proof}

Let $\mathcal{U}$ be any 
 ultrafilter on $\bB$ extending $\mathcal{F}$.
Then  $(\mathcal{F},\ge)\le_T(\mathcal{U},\ge)$, since the identity map on $(\mathcal{F},\ge)$ is a Tukey map.
Hence, $(\mathcal{U},\ge)\equiv_T([\kappa]^{<\om},\sse)$.
\end{proof}

The next theorem follows immediately by an application of  the Balcar-Fran\v{e}k Theorem, which states that every infinite complete Boolean algebra has an independent subset of cardinality that of the algebra (see Theorem 13.6 in \cite{KoppelbergHB}).

\begin{thm}\label{cor.complete.top}
Every infinite complete Boolean algebra has an ultrafilter with maximum Tukey type.
\end{thm}

\begin{proof}
By the Balcar-Fran\v{e}k Theorem, 
every infinite  complete Boolean algebra $\bB$ has an independent family of cardinality $|\bB|$.
By Theorem \ref{thm.indepimpliestop}, there is an ultrafilter $\mathcal{U}$ on $\bB$ such that 
$(\mathcal{U},\ge)\equiv_T([|\bB|]^{<\om},\sse)$.
\end{proof}

We mention  the following theorem of Shelah  giving sufficient conditions for a Boolean algebra to have an independent family of maximal size.

\begin{thm}[Shelah, (Theorem 10.1 in \cite{KoppelbergHB})]\label{thm.Shelah10.1}
Assume $\kappa,\lambda$ are regular infinite cardinals such that $\mu^{<\kappa}<\lambda$ for every cardinal $\mu<\lambda$, and that $\bB$ is a Boolean algebra satisfying the $\kappa$-chain condition.
Then every $X\sse\bB$ of size $\lambda$ has an independent subset of $Y$ of size $\lambda$.
\end{thm}

Thus, if $\bB$ has the $\kappa$-chain condition, $|\bB|=\lambda$,  and for all $\mu<\lambda$, $\mu^{<\kappa}<\lambda$,
then  $\bB$  contains an independent subset of size $|\bB|$.
Theorem \ref{thm.indepimpliestop} then implies that $\bB$ has an ultrafilter with maximum Tukey type.

The next theorem shows that every ultrafilter on a 
  free Boolean algebra  has  maximum Tukey type.
Thus, for each infinite cardinal $\kappa$, the spectrum of the Tukey types of ultrafilters on $\Clop(2^{\kappa})$ is precisely $\{([\kappa]^{<\om},\sse)\}$.

\begin{fact}\label{prop.freealgchar}
For each infinite cardinal $\kappa$ and each ultrafilter $\mathcal{U}$ on $\Clop(2^{\kappa})$, 
$(\mathcal{U},{\ge})\equiv_T([\kappa]^{<\om},\sse)$.
\end{fact}

\begin{proof}
Let $\kappa$ be an infinite cardinal.
The basic clopen sets of Clop$(2^{\kappa})$ are the sets 
$c_s=\{f\in 2^{\kappa}:f\contains s\}$,
where $s$ is any function from a finite subset of $\kappa$ into $2$.
Given an $x\in 2^{\kappa}$,
define a map $f$ from $[\kappa]^{<\om}$ into the neighborhood base of $x$ by 
letting $f(F)=c_{x\re F}$.
Then $f$ is a cofinal and unbounded map, so the neighborhood base of $x$ has cofinal type exactly $([\kappa]^{<\om},\sse)$.
Thus, every point in the Stone space $2^{\kappa}$ has neighborhood base of maximum Tukey type.
By Stone duality, this implies that each ultrafilter on Clop$(2^{\kappa})$ is Tukey equivalent to $([\kappa]^{<\om},\sse)$.
\end{proof}

\begin{rem}
The Stone space of each free Boolean algebra is homogeneous; that is, given any infinite cardinal $\kappa$,
for any two ultrafilters $\mathcal{U},\mathcal{V}$ on Clop$(2^{\kappa})$, there is a homeomorphism from Ult$($Clop$(2^{\kappa}))$ onto itself mapping $\mathcal{U}$ to $\mathcal{V}$
(see Exercise 4, page 139 in \cite{KoppelbergHB}).
Since  a homeomorphism maps any neighborhood base of $\mathcal{U}$ cofinally to any neighborhood base of $\mathcal{U}$, and vice versa,
 $\mathcal{U}$ is Tukey equivalent to $\mathcal{V}$.
In fact, homogeneity of the Stone space of any Boolean algebra implies  all its ultrafilters have the same Tukey type.
However, this says nothing about what  that Tukey type is.
We shall see in  Section  \ref{sec.IntAlg}
that it is possible to have an interval algebra in which all the ultrafilters have the same Tukey type $\kappa$, which is not the maximum type if $\kappa$ is an uncountable cardinal.
\end{rem}

Next, we investigate the Tukey spectra of completions of free algebras.
By Theorem \ref{cor.complete.top}, the completion  of Clop$(2^{\kappa})$, denoted 
 r.o.(Clop$(2^{\kappa})$),  always has an ultrafilter of the maximum  Tukey type $([2^{\kappa}]^{<\om},\sse)$.
In particular, the Cohen algebra r.o.(Clop$(2^{\om})$) has an ultrafilter of type $([\mathfrak{c}]^{<\om},\sse)$. 
This leads us to the following question.

\begin{question}\label{q.rofree.below}
Do all the ultrafilters in  the completion of a free Boolean algebra  have maximum Tukey type?
\end{question}

In Theorem \ref{thm.cohennot}, we will rule out some 
 possible  Tukey types  below the top for all completions of  free algebras.
We begin with two  propositions in which certain completeness or chain condition hypotheses rule out certain Tukey types of ultrafilters.

\begin{prop}\label{prop.kcomplete}
If $\kappa\ge \om$ and  $\bB$ is a $\kappa^+$-complete atomless Boolean algebra, then  $\bB$ has no ultrafilters of Tukey type $(\kappa,\le)$.
\end{prop}

\begin{proof}
Let $\mathcal{U}$ be an ultrafilter on $\bB$ and suppose toward a contradiction that there is a strictly decreasing sequence $\lgl b_{\al}:\al<\kappa\rgl$ cofinal in, and thus generating, $\mathcal{U}$.
Without loss of generality, we may assume that for each limit ordinal $\gamma<\kappa$,
$b_{\gamma}=\bigwedge_{\al<\gamma} b_{\al}$ and that $b_0=\bone$.
Since $\lgl b_{\al}:\al<\kappa\rgl$ generates an ultrafilter, it follows that $\bigwedge_{\al<\kappa}b_{\al}=\bo$.
For each $\al<\kappa$, define $a_{\al}=b_{\al}\wedge -b_{\al+1}$.
Since $\bB$ is atomless, there are non-zero $a_{\al, 0},a_{\al, 1}$
partitioning $a_{\al}$.
Let $c_i=\bigvee_{\al<\kappa} a_{\al,i}$, for $i<2$.
Then $c_0\vee c_1=\bone$ and $c_0\wedge c_1=\bo$;
so exactly one of $c_0,c_1$ must be in $\mathcal{U}$.
But for each $i<2$, we have that $c_i\not\ge b_{\al}$ for all $\al<\kappa$.
Since $\lgl b_{\al}:\al<\kappa\rgl$ generates $\mathcal{U}$, this implies that neither of $c_0,c_1$ is in $\mathcal{U}$, contradiction.
\end{proof}

\begin{prop}\label{prop.kcc}
Let $\kappa$ be a regular uncountable cardinal.
If $\bB$ is $\kappa$-c.c.,
then for all $\lambda\ge\kappa$, $\bB$ has no ultrafilters of Tukey type $(\lambda,\le)$.
\end{prop}

\begin{proof}
$\bB$ is $\kappa$-c.c. implies there are no strictly decreasing chains of order type $\lambda$ in $\bB$ for any $\lambda\ge\kappa$. 
In particular, no ultrafilter in $\bB$ can be generated by a strictly decreasing chain of order type $\lambda$.
\end{proof}

The following fact is due to Isbell.
Recall that  a partial ordering $(Q,\le_Q)$ is {\em relatively complete} if every subset of $Q$ which is bounded from below has a greatest lower bound  in $Q$.

\begin{prop}[Isbell, \cite{Isbell65}]\label{prop.monotone}
Let $(P,\le_P)$ and $(Q,\le_Q)$ be directed partial orderings. If $Q$ is relatively complete, then $P\ge_T Q$ if and only if there exists a monotone map $f \colon P\ra Q$  which has cofinal range.
\end{prop}

\begin{thm}\label{thm.cohennot}
Let $\kappa$ be an infinite cardinal.
Then each  ultrafilter on the completion of the free algebra on $\kappa$ generators  is not Tukey reducible to 
$\prod_{i\le n}\kappa_i$ for any finite  collection of regular cardinals $\{\kappa_i:i\le n\}$.
\end{thm}

\begin{proof}
Let $\kappa$ be an infinite cardinal and  $\mathcal{U}$ be any ultrafilter on $\ro(\Clop(2^{\kappa}))$.
We will show that for any collection of finitely many regular cardinals $\kappa_i$, $i\le n$, $\prod_{i\le n}\kappa_i\not\ge_T\mathcal{U}$.
It suffices  to consider only infinite cardinals $\kappa_i$, as
all ultrafilters on $\ro(\Clop(2^{\kappa}))$ are nonprincipal since $\ro(\Clop(2^{\kappa}))$ is atomless.
Since $\mathcal{U}$ is an ultrafilter, it is upwards closed.
In particular, the directed partial ordering $(\mathcal{U},\ge)$ is  relatively complete.
Thus, by Proposition \ref{prop.monotone},
whenever $(P,\le_P)\ge_T(\mathcal{U},\ge)$, there is a monotone cofinal map witnessing this.

Suppose $n=0$ and let $\kappa_0$ be any regular infinite  cardinal.
If $\kappa_0\ge_T\mathcal{U}$, then by Proposition \ref{prop.monotone}, there is a monotone cofinal map from $\kappa_0$ into $\mathcal{U}$.
Any monotone map will take $\kappa_0$ to a decreasing  sequence in $\mathcal{U}$, which is either eventually constant or else has a strictly decreasing subsequence of cofinality $\kappa_0$.
Propositions \ref{prop.kcomplete} and \ref{prop.kcc} imply that $\mathcal{U}$ is  not generated by any strictly decreasing infinite sequence of members of $\mathcal{U}$.
Since $\mathcal{U}$ is not principal it cannot be generated by a single member.
Thus, $\kappa_0\not\ge_T\mathcal{U}$.

Now suppose that $n\ge 1$ and 
given any infinite regular cardinals $\kappa_i$, $i<n$, 
$\prod_{i<n}\kappa_i\not\ge_T\mathcal{U}$.
Let $\{\kappa_i:i\le n\}$ be any collection of infinite  regular cardinals, and without loss of generality, assume they are indexed in strictly increasing order.
 Suppose that $f \colon \prod_{i\le n}\kappa_i\ra\mathcal{U}$ is a monotone  map into $\mathcal{U}$.
Fix a  sequence $\bar{\al}\in\prod_{i<n}\kappa_i$.
$f$ maps the strictly increasing sequence $\lgl {\bar{\al}}^{\frown}\beta: \beta\in \kappa_n\rgl$ in
$\prod_{i\le n}\kappa_i$
  to a decreasing (not necessarily strictly decreasing) sequence in $\mathcal{U}$, since $f$ is monotone.
Since $\kappa_n$ is uncountable and $\mathcal{U}$ has no uncountable strictly decreasing sequences,  there is some  $\beta(\bar{\al})\in\kappa_n$  such that
 $f({\bar{\al}}^{\frown}\beta(\bar{\al}))=
\min\{f({\bar{\al}}^{\frown}\beta):\beta\in\kappa_n\}$.
Let  $\gamma$ denote
$\sup\{  \beta(\bar{\al}): \bar{\al}\in\prod_{i<n}\kappa_i\}$.
Since $|\prod_{i<n}\kappa_i|=\kappa_{n-1}<\kappa_n$, it follows that
$\gamma<\kappa_n$.
Since $f$ is monotone,
the $f$-image of $\prod_{i\le n}\kappa_i$ is bounded below in $\mathcal{U}$ by the set $\{f(\bar{\al}^{\frown}\gamma):\bar{\al}\in\prod_{i<n}\kappa_i\}$.
By the induction hypothesis, this is not possible, since the set 
$\{(\bar{\al})^{\frown}\gamma:\bar{\al}\in\prod_{i<n}\kappa_i\}$ is isomorphic as a partially ordered set to $\prod_{i<n}\kappa_i$.
Thus, $f$ cannot be a cofinal map.

Therefore, there is no monotone cofinal map from $\prod_{i\le n}\kappa_i$ into $\mathcal{U}$. 
Hence, by 
Proposition \ref{prop.monotone},  
 $\prod_{i\le n}\kappa_i\not\ge_T\mathcal{U}$.
\end{proof}

In particular,  the Cohen algebra has  no ultrafilters of Tukey type $1$, $\om$, $\om_1$,  $\mathfrak{c}$, $\om\times\om_1$, or any finite product of regular cardinals $\kappa_i$, $i\le n$, where $\om\le \kappa_0<\kappa_1<\dots<\kappa_n\le \mathfrak{c}$.

\begin{question}
Does $\ro(\Clop(2^{\kappa}))$ have an ultrafilter Tukey equivalent to $\prod_{i\in I}\kappa_i$ for some infinite collection of regular infinite cardinals?
In particular, does the Cohen algebra have an ultrafilter Tukey equivalent to $(\om^{\om},\le)$?
\end{question}

\begin{rem}
Milovich in \cite{Milovich12} defines a preorder to be \textit{cofinally rectangular} if it is cofinally equivalent to the product of some finite  collection of regular cardinals. Theorem \ref{thm.cohennot} shows that all  ultrafilters on the completion of a free Boolean algebra (on infinitely many generators) are not cofinally rectangular. 
We will show in 
Section \ref{sec.IntAlg}, Theorem \ref{interval algebras}, 
 that there are Boolean algebras (namely interval algebras) which have only cofinally rectangular ultrafilters.
\end{rem}

The following simple fact shows that each finite-cofinite Boolean algebra has Tukey spectrum of size two, consisting exactly of the minimum and the maximum Tukey types.

\begin{fact}\label{prop.fincofin}
Let $X$ be any infinite set, and let $\bB$ denote the finite-cofinite algebra on $X$.
Then the  ultrafilters on $\bB$ consist exactly of the  principal ultrafilters and the cofinite ultrafilter.
Thus, the Tukey spectrum of $\bB$ is $\{(1,\le), ([\abs{X}]^{<\om},\sse)\}$.
\end{fact}

\begin{proof}
The ultrafilter of cofinite subsets of $X$ is isomorphic to $[|X|]^{<\om}$.
If an ultrafilter on $\bB$ contains a finite set, then it is a principal ultrafilter.
\end{proof}

This section closes with the Tukey spectra of Boolean algebras  generated by some almost disjoint family on an infinite set.
Let $\lambda$ be an infinite  cardinal.
A family $\mathcal{A}\sse\mathcal{P}(\lambda)$ is
{\em almost disjoint}
if  for all pairs $a,b\in\mathcal{A}$,  $|a\cap b|<\om$.
Given an almost disjoint family $\mathcal{A}\sse\mathcal{P}(\lambda)$,
the {\em almost disjoint Boolean algebra generated by $\mathcal{A}$} is 
the subalgebra  of $\mathcal{P}(\lambda)$ generated by $\mathcal{A}\cup [\lambda]^{<\om}$.
Any Boolean algebra generated from an almost disjoint family is superatomic. 
(See Example 0.1 on page 721 in \cite{RoitmanHBBAVol3}.)

\begin{thm}\label{thm.adfamily}
Let $\mathcal{A}$ be an almost disjoint family on an infinite cardinal $\lambda$ with $|\mathcal{A}|=\kappa$, 
and let $\bB$ denote the  subalgebra  of $\mathcal{P}(\lambda)$ generated by $\mathcal{A}$.
Then every ultrafilter on $\bB$  has Tukey type $1$ or  $[\mu]^{<\om}$ for some $\om\le \mu\le\kappa$.
The minimum and maximum types, $1$ and $[\kappa]^{<\om}$, are  always realized.
For $\om\le\mu<\kappa$,  $[\mu]^{<\om}$ is realized as the Tukey type of some ultrafilter on $\bB$ if and only if 
 there is an $a\in\mathcal{A}$ such that the set $\{a\cap-b:b\in\mathcal{A}\setminus\{a\}\}$ has cardinality $\mu$.
\end{thm}

\begin{proof}
Let $\mathcal{A}$ be an almost disjoint family on an infinite cardinal $\lambda$.
Let $\bB$ denote the subalgebra of $\mathcal{P}(\lambda)$ generated by $\mathcal{A}$, and let $\kappa=|\bB|$. 
The ultrafilters on $\bB$ are of three possible forms, each of which is realized:
 principal,  generated by $\{-a:a\in\mathcal{A}\}$, or generated by $\{a\cap - b: b\in\mathcal{A}\setminus \{a\}\}$ for some $a\in\mathcal{A}$.
The principal ultrafilters have Tukey type $1$.

Suppose 
 $\mathcal{U}$ is  the ultrafilter generated by the set $\{-a:a\in\mathcal{A}\}$.
In this case, index the members of $\mathcal{A}$ so that $\mathcal{A}=\{a_{\al}:\al<\kappa\}$.
This ultrafilter is nonprincipal.
As above,   the map $g \colon [\kappa]^{<\om}\ra\mathcal{U}$, given by $g(F)=\bigcap_{\al\in F}-a_{\al}$,  is a monotone cofinal Tukey map.
Therefore, $(\mathcal{U},\ge)\equiv_T([\kappa]^{<\om},\sse)$, since $|\bB|=\kappa$.

For the third type of ultrafilter, fix any $a\in\mathcal{A}$ and enumerate $\mathcal{A}\setminus \{a\}$ as $\{a_{\al}:\al<\kappa\}$, and let $\mathcal{U}$ be the ultrafilter generated by $\{a\}\cup\{-a_{\al}:\al<\kappa\}$.
Let $K\sse\kappa$ be a maximal subset of $\kappa$ such that for all $\al\ne\beta$ in $K$,
$a\cap -a_{\al}\ne a\cap -a_{\beta}$, and let $\mu=|K|$. 
Without loss of generality, 
suppose $\U$ is nonprincipal, and hence $\mu\ge \om$.
Define a map $g \colon [K]^{<\om}\ra \mathcal{U}$ by
$g(F)=a\cap(\bigcap_{\al\in F}-a_{\al})$, for each $F\in[K]^{<\om}$.
We claim that $g$ is a monotone cofinal and Tukey map.
By definition, it is clear that $g$ is monotone. Since its range is all of $\mathcal{U}$, it is a cofinal map.
To check that $g$ is Tukey,
let $\mathcal{X}\sse [K]^{<\om}$ and suppose that 
there is  a bound $b\in\mathcal{U}$ for $\{g(F):F\in\mathcal{X}\}$.
Then  $b\sse a\cap(\bigcap_{\al\in F}-a_{\al})$, for each $F\in\mathcal{X}$.
Letting  $G=\bigcup \mathcal{X}$, we see that
$b\sse a\cap (\bigcap_{\al\in G}-a_{\al})$.
Thus, $G$ must be finite; hence also $\mathcal{X}$ is finite and therefore bounded.
Therefore, $\mathcal{U}\equiv_T[\mu]^{<\om}$.
\end{proof}

\begin{question}
What are the Tukey spectra of superatomic Boolean algebras in general?
\end{question}

Lastly, we state a theorem that will be proved in Section 4: A tree algebra $\treealg T$ of size $\kappa$ has an ultrafilter with maximum Tukey type if and only if the underlying tree $T$ has an initial chain with $\kappa$-many immediate successors. (This is Corollary \ref{U top iff kappa succs of C}; we will also prove in Corollary \ref{p.when pseudo-tree algebras have top type ult} a more general version for pseudo-tree algebras.)

\section{Spectra of Tukey types of interval algebras, tree algebras, and pseudo-tree algebras}\label{sec.IntAlg}

Basic facts about interval algebras and tree algebras can be found in Volume 1 of the Handbook of Boolean Algebras \cite{KoppelbergHB}; basic facts about pseudo-tree algebras which generalize from corresponding tree algebra facts can be found in \cite{Koppelberg/Monk92}. Since pseudo-trees are probably the least well-known of these classes, we provide some background on them here. We follow the notation in \cite{Koppelberg/Monk92}. 

A \textit{pseudo-tree} is a partially ordered set $(T, \leq)$ such that for each $t \in T$, the set $T \downarrow t = \{s \in T: s \leq t\}$ is linearly ordered. The pseudo-tree algebra on a pseudo-tree $T$ is generated in the same way as a tree algebra: $\treealg(T)$ is the algebra of sets generated by the ``cones'' $T \uparrow t=\{s \in T: s \geq t\}$, for $t \in T$. A pseudo-tree algebra is thus a generalization of both an interval algebra and a tree algebra, and the following discussion of the correspondence between ultrafilters and initial chains applies to all three classes of Boolean algebras.

Let $T$ be an infinite pseudo-tree with a single root. (It does no harm to assume that all of our pseudo-trees have single roots; any pseudo-tree algebra is isomorphic to a pseudo-tree algebra on a pseudo-tree with a single root (see \cite{Koppelberg/Monk92}).) An initial chain in $T$ is a non-empty chain $C \subseteq T$ such that if $c \in C$ and $t < c$ then $t \in C$. There is a one-to-one correspondence between ultrafilters $\U$ on $\treealg T$ and initial chains in $T$, given by 
\[\phi(\U)=\{t \in T: T \cn t \in \U\}.\] 
The inverse of this map shows how ultrafilters are generated by sets defined in terms of their corresponding initial chains: for $C$ an initial chain in $T$, 
$\phi^{-1}(C)=\langle H_C \rangle,$ 
where \[H_C=\{(T \cn t) \setminus \bigcup_{s \in S}(T \cn s): t \in C,\ S \textrm{ is a finite antichain of elements }s > C\}\] (where $S$ is allowed to be empty).
This set of generators is closed under finite intersection, and 
$(\U, \supseteq) \equiv_T (H_C, \supseteq)$.

Let $C \subseteq T$ be an initial chain. Call a subset $R \subseteq T$ a \textit{set of approximate immediate successors of} $C$ if

(i) $r>C$ for all $r \in R$, and 

(ii) for all $s > C$, there is an $r \in R$ such that $C<r\leq s$.

\n Then define \[\varepsilon_C=\min\{\abs{R}: R \textrm{ is a set of approximate immediate successors of } C\}.\]

The character $\chi$ of an ultrafilter $\U$ on $\treealg T$ is the minimum size of a generating set for $\U$. If $C$ is the initial chain corresponding to $\U$, then $\chi(\U)=\max\{\varepsilon_C, \cf C\}$ (see \cite{Brown15}). If $T$ is a tree, then the set $\imm(C)$ of immediate successors of an initial chain $C$ in $T$ is well-defined, even if $C$ does not have a top element, and $\varepsilon_C=\abs{\imm(C)}$. 

We describe all possible Tukey types of ultrafilters on interval algebras in Theorem \ref{interval algebras}. 
From this, we classify the spectra of Tukey types of all interval algebras. In the terminology of Milovich \cite{Milovich12}, ultrafilters on interval algebras have cofinally rectangular Tukey types. We assume that all linear orders mentioned have least elements. (Where a linear order does not naturally have a least element, we add an element $-\infty$ to $L$ and proceed to build $\intalg L$ as in \cite{KoppelbergHB}.)

For $L$ a linear order and $X \subseteq L$, the coinitiality of $X$, denoted $\ci(X)$, is the least cardinal $\mu$ such that $\mu^*$ is coinitial in $X$.

\begin{thm}\label{interval algebras}
Let $L$ be a linear ordering.
Let $P$ denote the set of pairs of regular cardinals $(\kappa,\mu)$ for which there is an initial chain $C$ in $L$  such that the cofinality of $C$ is $\kappa$ and the coinitiality of $L\setminus C$ is $\mu$.
Then the Tukey spectrum of Intalg $L$ is exactly 
$\{(\kappa\times\mu,\le):(\kappa,\mu)\in P\}$.
\end{thm}

\begin{proof}
Let $L$ be a linear order with a first element and set $A=\intalg L$. Since $A$ is also a pseudo-tree algebra, its ultrafilters are associated with initial chains as are those of pseudo-tree algebras.
 Let $\U \in \ult A$ and let $C$ be the initial chain associated with $\U$. 
Let $\kappa =\cf C$ and  $\{c_\alpha:\alpha < \kappa\}$ be an increasing cofinal sequence in $C$.
Let
$\mu=\ci(L\setminus C)$ and let 
 $\{l_\beta:\beta < \mu\}$ be a decreasing coinitial sequence in $L \setminus C$. 
Letting 
\[G=\{[c_\alpha, l_\beta): \alpha < \kappa, \beta < \mu\},\] 
we see that 
 $\U=\langle G \rangle$, and hence $(\U, \supseteq) \equiv_T 
(G, \supseteq)$.
Define $f \colon (G, {\supseteq}) \to (\kappa \times \mu)$ by $f([c_\alpha, l_\beta))=(\alpha, \beta)$. One can check that $f$ is an unbounded cofinal map, so that $(G, \supseteq) \equiv_T (\kappa\times\mu,\le)$. 
\end{proof}

Thus if $L$ is a linear order and $\U$ is an ultrafilter on $\intalg L$, then there are only three possibilities for the Tukey type of $(\U, \supseteq)$. Letting $C$ be the initial chain corresponding to $\U$, either  
\begin{enumerate}
\item $(\U, \supseteq) \equiv_T (\kappa, \leq)$ where $\cf C = \kappa$, or
\item $(\U, \supseteq) \equiv_T (\mu, \leq)$ where   $\ci (L \setminus C) = \mu$, or
\item $(\U, \supseteq) \equiv_T (\kappa \times \mu, \leq)$ where $\cf C = \kappa$ and $\ci (L \setminus C) = \mu$.
\end{enumerate}

(Recall from Fact \ref{facts.useful} (5) that if $\cf C = \kappa = \ci (L \setminus C)$ then $(\U, \supseteq) \equiv_T (\kappa \times \kappa, {\leq}) \equiv_T (\kappa, \leq)$.)

Observe that Fact \ref{facts.useful} (4)  then implies that no uncountable interval algebra $\intalg L$ has an ultrafilter of maximal Tukey type $([\abs{L}]^{<\omega}, \subseteq)$. Countably infinite linear orders always have ultrafilters of top Tukey type:\ let $L$ be a countably infinite linear order. Then $L$ contains a initial chain $C$ with $\cf C = \omega$ or $\ci (L \setminus C) = \omega$; and in any case, if $\U$ is the ultrafilter corresponding to $C$, $(\U, \supseteq) \equiv_T (\omega, \leq) \equiv_T ([\omega]^{<\omega}, \subseteq)$.

Next, we show that  the Tukey spectra of  interval algebras is quite robust.

\begin{fact}\label{f.attainable types for interval algebras}
Given any collection  $P$ of pairs of regular cardinals, each of which is either $1$ or else is infinite, there is a linear order $L$ whose Tukey spectrum
 includes $\{(\kappa\times\mu,\le):(\kappa,\mu)\in P\}$.
\end{fact}

\begin{proof}
Let $\lambda=|P|$ and enumerate the pairs of regular cardinals in $P$ as $\{(\kappa_\alpha, \mu_\alpha): \alpha < \lambda\}$. 
For $\alpha < \lambda$, let $X_\alpha$ be a sequence of order type $\kappa_\alpha$, and let $Y_\alpha$ be a sequence of order type $\mu_\alpha^*$. Let
 $L$ be the linear order  
$$L=\bigcup_{\al<\lambda}\{\al\}\times
({X_{\al}}^{\frown}Y_{\al})$$
with the lexicographic ordering.
For $\alpha < \lambda$, let $C_\alpha$ denote the set of 
those elements of $L$ that are below every element of $Y_\alpha$. 
Then $\cf C_\alpha = \cf X_\alpha = \kappa_\alpha$ and $\ci(L \setminus C_\alpha)=\ci Y_\alpha = \mu_\alpha$. 
Letting $\U_\alpha$ be the ultrafilter corresponding to $C_\alpha$, it follows that $(\U_\alpha, \supseteq) \equiv_T (\kappa_\alpha \times \mu_\alpha, \leq)$. 
\end{proof}

We note that the Tukey spectrum of the interval algebra in the proof of Fact \ref{f.attainable types for interval algebras} may also contain types not among $\{(\kappa \times \mu, \leq): (\kappa, \mu) \in P\}$. For example, if $P=\{(\omega_1, \omega_1)\}$, then the construction gives us $L=\omega_1 + \omega_1^*$, and $\intalg(L)$ has the type $(\omega, \leq)$ in its Tukey spectrum.

Now we attend to the class of tree algebras. All products we mention in what follows will be of the following weak sort:\ for an index set $I$ and a collection of cardinals $\{\kappa_i:i \in I\}$,
let $\prod^{\mathrm{w}}_{i\in I}\kappa_i$ denote the collection of all functions $f\in\prod_{i\in I}\kappa_i$ such that for all but finitely many $i\in I$, $f(i)=0$; again the partial ordering is coordinate-wise. In the special case where the index set $I$ is a cardinal and each $\kappa_i$ is 2, we have the following fact.

\begin{fact}\label{f: top type equiv to a prod}
For any $\kappa$, $\prod^{\mathrm{w}}_{\alpha < \kappa}\{0,1\}\equiv_T \topkappa$.
\end{fact}

\begin{proof} For $F \in [\kappa]^{<\om}$, set $f(F)=\langle e_\alpha: \alpha < \kappa \rangle$ where 

\[ e_\alpha = \begin{cases}
1 \textrm{ if } \alpha \in F\\
0 \textrm{ if } \alpha \not\in F
\end{cases}
\] 

Then $f$ is an unbounded cofinal map.
\end{proof}

Proposition \ref{treealg of size om has top type ult} and Corollary \ref{U top iff kappa succs of C} characterize those trees $T$ for which  the algebra $\treealg T$ has an ultrafilter of maximal Tukey type $([\abs{T}]^{<\omega},\subseteq)$.

\begin{prop}\label{treealg of size om has top type ult}
If $T$ is a tree of size $\omega$ with a single root, then $\treealg T$ has an ultrafilter of type $([\om]^{<\om}, \subseteq)$.
\end{prop}

\begin{proof}
Let $T$ be a tree of size $\om$ with a single root. Then also $\abs{\treealg T}=\om$.
If there is  some $z \in T$ with $\omega$-many immediate successors $\{s_i:i<\om\}$,
then set $C=T \bl z$ and let $\U=\phi^{-1}(C)$ be the corresponding ultrafilter. Define $a_n \in \treealg T$, for $n < \om$, by $a_n=(T \cn z) \setminus \bigcup_{i<n} (T \cn s_i)$. Then the $a_n$ form a chain of type $\om$ in $(\U, \supseteq )$, so that $(\U, \supseteq) \equiv_T (\om, \leq) \equiv_T ([\omega]^{<\omega}, \subseteq)$.

Otherwise, no $z \in T$ has $\omega$-many immediate successors. Since $T$ has a single root, this means that for all $n<\omega$, $\abs{\operatorname{Lev}_n(T)}<\om$; that is, $T$ is an $\om$-tree. By K\"{o}nig's Lemma, $T$ has an infinite chain $\{c_n:n<\om\}$. Let $C$ be minimal among initial chains containing $\{c_n:n<\om\}$. Let $\U=\phi^{-1}(C)$ be the ultrafilter corresponding to $C$. For each $n<\om$, set $a_n=T \cn c_n$. Then the $a_n$ form a chain of type $\om$ in $( \U, \supseteq)$, so that $(\U, \supseteq) \geq_T (\om, \leq)$. Since $(\om, \leq)$ has the maximum Tukey type for $\om$, $(\U, \supseteq) \equiv_T ( [\omega]^{<\omega}, \subseteq )$.
\end{proof}

\begin{thm}\label{tree algebras}
Let $T$ be a tree, let $\U$ be an ultrafilter on $T$, and let $C$ be the initial chain corresponding to $\U$. Let $\{c_\alpha: \alpha < \cf C\}$ be an increasing cofinal sequence in $C$, and let $\{s_\beta: \beta < \mu\}$ be the set of immediate successors of $C$ in $T$. Then $(\U, \supseteq) \equiv_T (\cf C \times \prod^\mathrm{w}_{\beta < \mu}\{0,1\}, \leq)$.
\end{thm}

\begin{proof}
$\U$ is generated by $G=\{(T \cn c_\alpha) \setminus \bigcup_{\beta \in F}(T \cn s_\beta): \alpha < \cf C, F \in [\mu]^{<\omega}\}$. Define $f \colon (G, \supseteq) \to (\cf C \times \prod^\mathrm{w}_{\beta < \mu}\{0,1\}, \leq)$ by $f((T \cn c_\alpha) \setminus \bigcup_{\beta \in F}(T \cn s_\beta))=\langle \alpha \rangle^{\smallfrown} \langle e_\beta: \beta < \mu\rangle$ where 
\[ e_\beta = \begin{cases}
1 \textrm{ if } \beta \in F\\
0 \textrm{ if } \beta \not\in F.
\end{cases}
\]
We claim that $f$ is an unbounded cofinal map. Let $X \subseteq G$ be unbounded. Then either (i) the set \[\{\alpha < \cf C: (T \cn c_\alpha) \setminus \bigcup_{\beta \in F}(T \cn s_\beta) \in X \textrm{ for some } F \in [\mu]^{<\omega}\}\] is unbounded in $\cf C$, or (ii) the set 
 \[\{\gamma < \mu: (T \cn c_\alpha) \setminus \bigcup_{\beta \in F}(T \cn s_\beta) \in X \textrm{ for some } \alpha < \cf C, F \in [\mu]^{<\omega} \textrm{ with } \gamma \in F\}\]
 is infinite. In case (i), the set of first coordinates of elements of $f[X]$ is unbounded in $\cf C$. In case (ii), there are infinitely many $\beta < \mu$ at which some element of $f[X]$ has the value 1. In either case, $f[X]$ is unbounded in $(\cf C \times \prod^\mathrm{w}_{\beta < \mu}\{0,1\}, \leq)$.
 
 Now suppose $X \subseteq G$ is a cofinal subset. Let $p=\langle \alpha, e_0, e_1, \dots, e_\beta, \dots \rangle \in (\cf C \times \prod^\mathrm{w}_{\beta < \mu}\{0,1\}, \leq)$. Set $F=\{\beta < \mu: e_\beta = 1\}$. As $X$ is cofinal in $G$, there is some $x \in X$ such that $x \subseteq (T \cn c_\alpha) \setminus \bigcup_{\beta \in F}(T \cn s_\beta)$. Then $f(x) \geq p$. Thus $f[X]$ is cofinal in $(\cf C \times \prod^\mathrm{w}_{\beta < \mu}\{0,1\}, \leq)$.
 
 Then since $G$ is closed under finite intersection and generates $\U$, \[(\U, \supseteq) \equiv_T (G, \supseteq) \equiv_T (\cf C \times \prod^\mathrm{w}_{\beta < \mu}\{0,1\}, \leq).\]
\end{proof}

From Theorem \ref{tree algebras} and Fact \ref{f: top type equiv to a prod}, we have the following corollary:

\begin{cor}\label{U top iff kappa succs of C}
Let $\kappa > \om$ and let $T$ be a tree of size $\kappa$. Then $\treealg T$ has an ultrafilter with type $\topkappa$ if and only if $T$ has an initial chain with $\kappa$-many immediate successors.
\end{cor}

From Theorem \ref{tree algebras} we can describe the Tukey spectrum of a tree algebra.

\begin{thm} \label{c.spectrum of tree algebra}
Let $T$ be a tree. The Tukey spectrum of $\treealg T$ consists of those Tukey types $(\cf C \times \prod^\mathrm{w}_{\beta < \mu}\{0,1\}, \leq)$ where $C \subseteq T$ is an initial chain with $\mu$-many immediate successors. 
\end{thm}

Finally, we turn our attention to the broader class of pseudo-tree algebras.
Corollary \ref{U top iff kappa succs of C} showed that the tree algebra on an uncountable  tree $T$ has an ultrafilter of top Tukey type if and only if there is an initial chain $C$ in $T$ with $\abs{\imm(C)}=\abs{T}$.
One possible pseudo-tree analog of ``having an initial chain with $\abs{T}$-many immediate successors'' would be ``having an initial chain $C$ with $\varepsilon_C=\abs{T}$'', and so one could ask whether the ultrafilter corresponding to such an initial chain has top Tukey type. The answer is ``no''. For example, let $T$ be a linear order of order type $\omega_1 + 1 + \omega_1^*$ (this is also a pseudo-tree), and let $C$ consist of the first $\omega_1$-many elements of $T$. Then by Theorem \ref{interval algebras}, the ultrafilter corresponding to $C$ has type $(\omega_1, \leq)$, which is strictly less than the top type for $\treealg T$ by Fact \ref{facts.useful} (4).

The next Fact shows that uncountable pseudo-tree algebras always have ultrafilters whose Tukey type is strictly below the maximum type.

\begin{fact}\label{f.unctbl tree algebras have more than top type}
Let $T$ be an uncountable pseudo-tree. Then there is an ultrafilter $\U$ on $\treealg T$ whose Tukey type is strictly less than that of $([\abs{T}]^{<\omega}, \subseteq)$.
\end{fact}

\begin{proof}
Let $T$ be an uncountable pseudo-tree. Let $C$ be any initial chain such that for no $t \in T$ is $t > C$, and let $\U$ be the ultrafilter corresponding to $C$. Then $C$ is generated by $G=\{T \cn c: c \in C\}$, so that $(U, \supseteq) \equiv_T (G, \supseteq) \equiv_T (\kappa, \leq)$ for some cardinal $\kappa$. Since $T$ is uncountable, $(\kappa, \leq) \not\equiv_T ([\abs{T}]^{<\omega}, \subseteq)$.
\end{proof}

\begin{prop}\label{finite product types attained}
Let $\lambda$ be any cardinal (finite or infinite) and let $\{\kappa_\delta: \delta < \lambda\}$ be a set of 
regular 
cardinals, each of which is either infinite or else $2$. 
Then there is a pseudo-tree $T$ and an ultrafilter $\U$ on $\treealg T$ such that $(\U, \supseteq) \equiv_T \prod^{\mathrm{w}}_{\delta < \lambda} \kappa_\delta$.
\end{prop}

\begin{proof}
Let 
 regular
cardinals $\kappa_\delta$, for $\delta < \lambda$, be given, where each $\kappa_{\delta}$ is either infinite or else $2$.  
For each $\delta<\lambda$, 
let $\kappa'_\delta=\kappa_\delta$ if $\kappa_\delta$ is an infinite cardinal, and let $\kappa'_\delta=1$ if $\kappa_\delta=2$.
Let $T$ be the pseudo-tree constructed by putting above a single root $r$ $\lambda$-many pairwise-incomparable linear orders $T_\delta=\{t^\delta_\alpha: \alpha < \kappa'_\delta\}$ where each $T_\delta$ is an inverted copy of $\kappa'_\delta$ -- that is, $T_\delta$ is isomorphic to $(\kappa'_\delta)^*$.
 Let $A=\treealg T$ be the pseudo-tree algebra on $T$. Let $C=\{r\}$ and let $\U$ be the ultrafilter associated with $C$. 
Then $\U=\langle G \rangle$ where \[G=\{(T \cn r) \setminus \bigcup_{s \in S}(T \cn s): S \textrm{ is a finite antichain of elements above } r\}.\] 
Since any such antichain $S$ consists of at most one element $t^\delta_{\alpha_\delta}$ from each $T_\delta$, for $\delta < \lambda$, a typical element of $G$ is of the form 
$(T \cn r) \setminus \bigcup_{\delta\in F}(T \cn t^{\delta}_{\alpha_{\delta}})$ for some  $F\in [\lambda]^{<\om}$ and  $\alpha_{\delta}<\kappa'_{\delta}$, $\delta\in F$. 
Also note that $(G, \supseteq) \equiv_T (\U, \supseteq)$, since $G$ is cofinal in $\U$.

Define $f \colon (G, \supseteq) \to (\prod^{\mathrm{w}}_{\delta < \lambda}\kappa_\delta, \leq)$ by \[f\left( (T \cn r) \setminus \bigcup_{\delta\in F}(T \cn t^{\delta}_{\alpha_{\delta}}) \right)=\langle \beta_\delta: \delta < \lambda \rangle\] where, for $\delta < \lambda$, 
\[
\beta_\delta = \begin{cases}
0 & \textrm{ if } \delta \not\in F\\
1+\alpha_{\delta} & \textrm{ if } \delta \in F
\end{cases}
\]
(That is:\ in those coordinates where no part of the chain $T_\delta$ has been excluded, let  $\beta_\delta=0$; and in those coordinates where $T_\delta$ has been cut-off at $t^{\delta}_{\alpha_{\delta}}$, let $\beta_\delta=1+\alpha_{\delta}$.)
Then $f$ is a bijection between $(G, \supseteq)$ and  $(\prod^{\mathrm{w}}_{\delta < \lambda}\kappa_\delta, \leq)$ such that for all $g,g'\in G$, $g\le g'\lra f(g)\le f(g')$.
By Fact \ref{fact.mono} (2), 
 $(\U, \supseteq) \equiv_T (G, {\supseteq}) \equiv_T (\prod^{\mathrm{w}}_{\delta < \lambda}\kappa_\delta, \leq)$.
\end{proof}

The following Lemma \ref{lem.lambdafan} distills the essential structure inside any set of approximate immediate successors above a chain in a pseudo-tree.
This structure is given by the following notion of a $\lambda$-fan.

\begin{defn}\label{def.lambdafan}
Let   $C$ be an initial chain in a pseudo-tree $T$.
A set $\Lambda=\{t^\alpha_\beta: \alpha < \lambda, \beta < \theta_\alpha\}$   of elements of $T$ above $C$
  is a \textit{$\lambda$-fan above $C$} if  the following four conditions hold:
\begin{enumerate}
\item
$\Lambda$ is a set of approximate immediate successors of $C$.
\item
$\lambda$ is a cardinal and for each $\alpha<\lambda$, $\theta_\alpha$ is either equal to $1$ or an infinite regular cardinal.
\item
For each $\al<\lambda$, 
$\{t^\alpha_\beta:  \beta < \theta_\alpha\}$
 is a strictly decreasing coinitial chain  above $C$.
\item
For $\alpha < \alpha' < \lambda$, $\beta < \theta_\alpha$, and $\beta' < \theta_{\alpha'}$, $t^\alpha_\beta \perp t^{\alpha'}_{\beta'}$. 
\end{enumerate}
\end{defn}

\begin{lem}\label{lem.lambdafan}
Let $T$ be a pseudo-tree and  let $C$ be an initial chain in $T$.
Given any set $S$ of approximate immediate successors of $C$,
there is a $\lambda$-fan $\Lambda=\{t^{\al}_{\beta}: \al<\lambda, \beta<\theta_{\al}\}$ above $C$ which is coinitial in $S$.
Moreover, the cardinal $\lambda$ and set of cardinals $\{\theta_{\al}:\al<\lambda\}$ are uniquely determined by $C$.
\end{lem}

\begin{proof}
Let $T$ be a pseudotree and let 
 $C$ be an initial chain in $T$.
Let $S\sse T$ be a set of approximate immediate successors of $C$ of minimal cardinality.
Let $\kappa=|S|$, and enumerate the members of $S$ as $s_{\delta}$, for $\delta<\kappa$.
For each $\delta<\kappa$,
let $\gamma(\delta)$ be the least $\gamma<\kappa$ such that 
$((T \bl s_{\gamma}) \setminus C)\cap((T \bl s_\delta) \setminus C)\ne\emptyset$.
Note that $\gamma(\delta)\le\delta$.
Observe that 
each $(T \bl s_\delta) \setminus C$ is ``closed downwards above $C$'' -- that is, if $t \in (T \bl s_\delta) \setminus C$ and $C < s \leq t$, then $s \in (T \bl s_\delta) \setminus C$.

\begin{subclaim} $((T \bl s_\delta)\setminus C)\cap((T \bl s_\eta)\setminus C)\ne \emptyset$
if and only if $\gamma(\delta)=\gamma(\eta)$.
\end{subclaim}

\begin{proof}
 First suppose $((T \bl s_\delta)\setminus C)\cap((T \bl s_\eta)\setminus C)\ne \emptyset$. Since also $((T \bl s_\delta)\setminus C)\cap((T \bl s_{\gamma(\delta)})\setminus C)\ne \emptyset$ and all of these sets are downwards-closed linear orders above $C$, we have $((T \bl s_\eta)\setminus C)\cap((T \bl s_{\gamma(\delta)})\setminus C)\ne \emptyset$. Then $\gamma(\eta) \leq \gamma(\delta)$ by minimality of $\gamma(\eta)$. Similarly $\gamma(\delta) \leq \gamma(\eta)$.

Now suppose that $\gamma(\delta) = \gamma(\eta)$. Then $((T \bl s_\delta)\setminus C)\cap((T \bl s_{\gamma(\delta)})\setminus C)\ne \emptyset$ and $((T \bl s_\eta)\setminus C)\cap((T \bl s_{\gamma(\delta)})\setminus C)\ne \emptyset$. It follows that $((T \bl s_\delta)\setminus C)\cap((T \bl s_\eta)\setminus C)\ne \emptyset$.
\end{proof}

By Subclaim 1,  $(T \bl s_\delta) \setminus C$ and $(T \bl s_\eta) \setminus C$ are disjoint if and only if 
$\gamma(\delta)\ne\gamma(\eta)$.
Let $\lambda$ be the cardinality of the set $\{\gamma(\delta):\delta<\kappa\}$, and enumerate the set $\{s_{\gamma(\delta)}:\delta<\kappa\}$ as $\{t^\al_0:\al<\lambda\}$.
For each $\al<\lambda$, let $\theta_{\al}$ be the coinitiality of $(T \bl t^{\al}_0) \setminus C$, and let $\{t^\al_\beta:  \beta < \theta_\al\}$ be a decreasing coinitial sequence in $(T \bl t^\al_0)\setminus C$.
It follows that  (2) - (4) in the definition of a $\lambda$-fan hold.

\begin{subclaim}
Let $\Lambda=\{t^\al_\beta: \al< \lambda, \beta < \theta_\al\}$. Then $\Lambda$ is a set of approximate immediate successors of $C$.
\end{subclaim}

\begin{proof}
  First note that $C < \Lambda$. Suppose $r>C$. As $S$ is a set of approximate immediate successors of $C$, there is an $s_\delta \in S$ with $C < s_\delta \leq r$.
We have $((T \bl s_\delta) \setminus C) \cap ((T \bl s_{\gamma(\delta)}) \setminus C) \neq \emptyset$. 
Say $s_{\gamma(\delta)}=t_0^{\delta}$ for some $\delta < \lambda$. 
Pick $\beta$ large enough so that $t^\al_\beta \in  ((T \bl s_\delta) \setminus C) \cap ((T \bl s_{\gamma(\delta)}) \setminus C)$. 
Then $t^\al_\beta \in \Lambda$ and $C < t^\al_\beta \leq s_\delta \leq r$. Thus $\Lambda$ is a set of approximate immediate successors of $C$; hence, Subclaim 2 holds. 
\end{proof}

Thus, (1) in the definition of $\lambda$-fan holds, so  $\Lambda$ is a $\lambda$-fan above $C$.

Now suppose $S'$ is another set of approximate immediate successors of $C$  and that  $\Lambda'=\{u^{\xi}_{\zeta}:\xi<\lambda',\ \zeta<\theta'_{\xi}\}\sse S'$ is 
 a $\lambda'$-fan  above $C$.
For each $\al<\lambda$, let $S_{\al}$ denote $\{t^{\al}_{\beta}:\beta<\theta_{\al}\}$ and for each $\xi<\lambda'$, let $S'_{\xi}$ denote $\{u^{\xi}_{\zeta}:\zeta<\theta'_{\xi}\}$.
Let $\overline{S}_{\al}$ denote $\{t\in T: t>C $ and $\exists \beta<\theta_{\al}(t\le t^{\al}_{\beta})\}$,
and  $\overline{S}_{\xi}'$ denote  $\{t\in T: t>C $ and $\exists \zeta<\theta_{\xi}'(t\le u^{\xi}_{\zeta})\}$.
For each $\al<\lambda$, there is a $\xi(\al)<\lambda'$ and a $\zeta(\al)<\theta'_{\xi(\al)}$ such that $u^{\xi(\al)}_{\zeta(\al)}\le t^{\al}_0$.
Thus,
 $\overline{S}_{\al}$ and $\overline{S}_{\xi(\al)}'$ 
have a common coinitial segment above $C$.
By property (4) of a $\lambda$-fan,
for all $\xi'\ne\xi(\al)$,
$\overline{S}'_{\xi'}\cap\overline{S}'_{\xi(\al)}=\emptyset$.
Therefore,
this  $\xi(\al)$ is unique.  
Define $\varphi(\al)$ to be this $\xi(\al)$.
This defines  a one-to-one function $\varphi:\lambda\ra\lambda'$.
Since for each $\xi<\lambda'$ there is an $\al(\xi)<\lambda$ and a $\beta(\xi)$ such that $t^{\al(\xi)}_{\beta(\xi)}\le u^{\xi}_0$,
a similar argument reveals that
the function $\varphi$ is also onto $\lambda'$.
Thus, $\varphi$ is a bijection; hence $\lambda=\lambda'$.
As noted above, for each $\al<\lambda$,
$\overline{S}_{\al}$ and $\overline{S}_{\varphi(\al)}'$ 
have a common coinitial segment above $C$.  
The coinitiality of this segment must simultaneously be equal to $\theta_{\al}$ and $\theta'_{\varphi(\al)}$.
Hence, $\theta'_{\varphi(\al)}=\theta_{\al}$.
Therefore, the cardinals representing any $\lambda$-fan above $C$ are uniquely determined by $C$.
\end{proof}

The previous lemma is now applied to characterize the Tukey types of ultrafilters on  pseudo-tree algebras.

\begin{thm}\label{Tukey types of ults on pseudo-tree algebras} 
Let $T$ be a pseudo-tree and let $\U$ be an ultrafilter on $\treealg T$. 
Then $\U$ is Tukey equivalent to $(\mu\times\prod^{\mathrm{w}}_{\al < \lambda}\kappa_\al, \leq)$,
where 
$\mu$ is either $1$ or an infinite regular cardinal,
$\lambda$ is some cardinal, and for $\al<\lambda$, $\kappa_\al$ is either $2$ or else an infinite regular cardinal.
\end{thm}

\begin{proof}
Let $T$ be a pseudotree and let $\mathcal{U}$ an ultrafilter on $\treealg T$.
Let $C=\phi(\U)$ be the initial chain corresponding to $\mathcal{U}$ and let $S\sse T$ be a set of approximate immediate successors of $C$ of minimal cardinality.
By Lemma
\ref{lem.lambdafan},
there is a coinitial subset $\Lambda\sse S$ which is a
  $\lambda$-fan above $C$ of the form $\{t^{\al}_{\beta}:\al<\lambda,\ \beta<\theta_{\al}\}$
which is coinitial in $S$,  where $\lambda$ is a cardinal, each $\theta_{\al}$ is either $1$ or  an infinite regular cardinal, and these cardinals are uniquely determined by $C$.
For each $\al<\lambda$, let $\kappa_{\al}=1 +\theta_{\al}$.
Shifting the lower indices of the members of $\Lambda$ by one and letting $s^{\al}_{1+\beta}=t^{\al}_{\beta}$,
we re-write  $\Lambda$ as $\{s^{\al}_\beta:\al<\lambda$, $1\le\beta<\kappa_{\al}\}$.

Let $\mu=\cf C$, and let
$\{c_\xi: \xi < \mu\}$ be an increasing cofinal sequence in $C$. 
Then  $\U=\langle G \rangle$ where 
\[G=\{(T \cn c_\xi)\setminus \bigcup_{s \in F}(T \cn s): \xi < \mu \textrm{ and $F$ is a finite antichain in $\Lambda$}\},\]
since $\Lambda$ is a set of approximate immediate successors of $C$. 
We claim that $(G, \supseteq) \equiv_T (\mu \times \prod^{\mathrm{w}}_{\al < \lambda}\kappa_\al, \leq)$. 
Define a map $f \colon G \to \mu \times \prod^{\mathrm{w}}_{\al < \lambda}\kappa_\al$ by 
$f((T \cn c_\xi)\setminus \bigcup_{s \in F}(T \cn s)) = \langle \xi \rangle^{\smallfrown} \langle \varepsilon_\al: \al < \lambda \rangle$, where for $\al<\lambda$,
\[
\varepsilon_\al = \begin{cases}
0, \textrm{ if } s^\al_\beta \not\in F \textrm{ for all }1 \leq \beta < \kappa_\al\\
\beta, \textrm{ if } s^\al_\beta \in F \textrm{ for some }1 \leq \beta < \kappa_\al.
\end{cases}
\]
Note that $f$ is a bijection which preserves order both ways; that is,  for $g,g'\in G$, $g\le g'\lra f(g)\le f(g')$.
Thus $(\U, \subseteq) \equiv_T (G, \supseteq) \equiv_T \mu \times \prod^{\mathrm{w}}_{\al < \lambda}\kappa_\al$,
by Fact \ref{fact.mono} (2). 
\end{proof}

Note that this is consistent with what we already know about ultrafilters on tree algebras from Corollary \ref{U top iff kappa succs of C}, because for any $\kappa$, $\topkappa \equiv_T \prod^{\mathrm{w}}_{\alpha < \kappa} \{0,1\}$.

Since by Lemma \ref{lem.lambdafan},  every set of approximate immediate successors of an initial chain $C$ has, up to isomorphism, the same $\lambda$-fan,
we shall speak of {\em the} $\lambda$-fan above $C$.
The characterization of the Tukey spectra of pseudo-trees  follows immediately from
Lemma \ref{lem.lambdafan} and
Theorem \ref{Tukey types of ults on pseudo-tree algebras}.
Further, combined with Fact \ref{facts.useful} (2), they tell us exactly when a pseudo-tree has an ultrafilter with maximum Tukey type.

\begin{thm}\label{p.when pseudo-tree algebras have top type ult}
Let $T$ be a pseudo-tree with a single root. The Tukey spectrum of $\treealg T$ consists of exactly those types $(\cf C \times \prod^{\mathrm{w}}_{\alpha < \lambda}\kappa_\alpha, \leq)$,
 where $C$ is an initial chain with  $\lambda$-fan  of approximate immediate successors  $\Lambda=\{s^\alpha_\beta: \alpha < \lambda, \beta < \theta_\alpha\}$
and $\kappa_{\alpha}=1+\theta_{\alpha}$.

If $\abs{T}=\kappa$, then $T$ has an ultrafilter with maximal type $([\kappa]^{< \omega}, \subseteq)$ if and only if there is an initial chain $C \subseteq T$
with a $\kappa$-fan $K=\{s^\alpha_\beta: \alpha < \kappa, \beta < \theta_\alpha\}$ of approximate immediate successors above $C$,  where $\kappa$-many of the cardinals $\theta_\alpha$ are 1.
\end{thm}

\begin{rem}
We point out that $\cf C \times \prod^{\mathrm{w}}_{\alpha < \lambda}\kappa_\alpha$ is really the same as $\prod^{\mathrm{w}}_{\alpha < \lambda+1}\kappa_\alpha$,
where we define $\kappa_{\lambda}$ to be $\cf C$.
We leave the $\cf C$ on the left to remind the reader of how the structure of the pseudo-tree influences the Tukey types of its ultrafilters. 
\end{rem}

\section{Questions}\label{sec.Questions}

We conclude with some open  questions either motivating or arising from this paper.

\begin{question}\label{q.WhichHaveTop}
Characterize those Boolean algebras that have an ultrafilter of the maximum  Tukey type, and characterize those Boolean algebras that have an  ultrafilter of  Tukey type strictly 
below the maximum.
\end{question}

\begin{question}\label{q.charfree}
If $\bB$ is an infinite  Boolean algebra  such  that all ultrafilters  on $\bB$ have maximum Tukey type, is $\bB$ necessarily a free algebra? 
\end{question}

Or is the following possible?

\begin{question}\label{q.TscFreealg}
Does the completion of a free Boolean algebra  have only  ultrafilters which are of maximum Tukey type?
\end{question}

If not, can we at least rule out the remaining possible case of a  cofinally scalene ultrafilter?

\begin{question}\label{q.CohenAlg}
Can the completion of a free Boolean algebra have an ultrafilter which is cofinally scalene; that is Tukey equivalent to $\prod_{\al<\lambda}\kappa_{\al}$, for some infinite $\lambda$ and each $\kappa_{\al}\ge \om$?
In particular, can the Cohen algebra have an ultrafilter Tukey equivalent to $(\om^{\om},\le)$?
\end{question}

Note that a positive answer to Question \ref{q.CohenAlg} would imply a negative answer to Question \ref{q.TscFreealg}, and a positive answer to Question \ref{q.TscFreealg} would imply a negative answer to Question \ref{q.charfree}.

\begin{question}\label{q.Superatomic}
What are the Tukey spectra of superatomic Boolean algebras?
\end{question}

\section{Bibliography}

\bibliographystyle{spmpsci}
\bibliography{references}

\end{document}